\documentclass[12pt]{article}
\usepackage{amssymb}
\usepackage{amsmath,amsthm}
\usepackage[latin1]{inputenc}
\usepackage{hyperref}
\usepackage{enumerate}
\usepackage{tikz}
\usepackage{tkz-graph}

\hypersetup{colorlinks=true}

\hypersetup{colorlinks=true, linkcolor=blue, citecolor=blue,urlcolor=blue}

 \newtheorem{remark}{Remark}

 \newtheorem{lemma}[remark]{Lemma}
 \newtheorem{theorem}[remark]{Theorem}
 \newtheorem{proposition}[remark]{Proposition}
 \newtheorem{corollary}[remark]{Corollary}
  \newtheorem{claim}[remark]{Claim}

\title{Closed formulae for the strong metric dimension of lexicographic product graphs}

\author{Dorota Kuziak$^{(1)}$, Ismael G. Yero$^{(2)}$ and Juan A. Rodr\'{\i}guez-Vel\'{a}zquez$^{(1)}$\\
$^{(1)}${\small Departament d'Enginyeria Inform\`atica i Matem\`atiques,}\\
{\small Universitat Rovira i Virgili,} {\small Av. Pa\"{\i}sos Catalans 26, 43007 Tarragona, Spain.}\\
{\small dorota.kuziak\@@urv.cat, juanalberto.rodriguez\@@urv.cat}\\
$^{(2)}${\small Departamento de Matem\'aticas, Escuela Polit\'ecnica Superior de Algeciras}\\
{\small Universidad de C\'adiz,} {\small Av. Ram\'on Puyol s/n, 11202 Algeciras, Spain.}\\
{\small ismael.gonzalez\@@uca.es}\\
}

\begin{document}
\maketitle

\begin{abstract}
Given a connected graph $G$, a vertex $w\in V(G)$ strongly resolves two vertices $u,v\in V(G)$ if there exists some shortest $u-w$ path containing $v$ or some shortest $v-w$ path containing $u$. A set $S$ of vertices is a strong metric generator for $G$ if every pair of vertices of $G$ is strongly resolved by some vertex of $S$. The smallest cardinality of a strong metric generator for $G$ is called the strong metric dimension of $G$. In this paper we obtain several relationships between the strong metric dimension of the lexicographic product of graphs and the strong metric dimension of its factor graphs.
\end{abstract}



\section{Introduction}

A vertex $v$ of a connected graph $G$ is said to distinguish two vertices $x$ and $y$ of $G$ if $d_G(v,x)\ne d_G(v,y)$, \textit{i.e.}, the distance between $v$ and $x$ is different from the distance between $v$ and $y$. A set $S\subset V(G)$ is said to be a \emph{metric generator} for $G$ if any pair of vertices of $G$ is distinguished by some element of $S$. A minimum generator is called a \emph{metric basis}, and its cardinality the \emph{metric dimension} of $G$. The problem of uniquely recognizing the position of an intruder in a network was the principal motivation of introducing the concept of metric generators in graphs by Slater in \cite{Slater1975}, where metric generators were called \emph{locating sets}. An analogous concept was also introduced independently by Harary and Melter in \cite{Harary1976}, where the metric generators were called \emph{resolving sets}. Several applications and theoretical studies about metric generators have been presented and published. In this sense, according to the amount of literature concerning this topic and all its close variants, we restrict our references to those ones which are only citing papers that we really refer to in a non-superficial way.

Another invariant, more restricted than the metric dimension, was presented by Seb\H{o} and Tannier in \cite{Sebo2004}, and studied further in several articles. That is, a vertex $w\in V(G)$ \emph{strongly resolves} two vertices $u,v\in V(G)$ if $d_G(w,u)=d_G(w,v)+d_G(v,u)$ or $d_G(w,v)=d_G(w,u)+d_G(u,v)$, \emph{i.e.}, there exists some shortest $w-u$ path containing $v$ or some shortest $w-v$ path containing $u$. A set $S$ of vertices in a connected graph $G$ is a \emph{strong metric generator} for $G$ if every two vertices of $G$ are strongly resolved by some vertex of $S$. The smallest cardinality of a strong metric generator for $G$ is called the \emph{strong metric dimension} and is denoted by $\dim_s(G)$. A \emph{strong metric basis} of $G$ is a strong metric generator for $G$ of cardinality $\dim_s(G)$. The strong metric dimension of product graphs has been previously studied for the case of Cartesian product graphs and direct product graphs \cite{Rodriguez-Velazquez2013a}, strong product graphs \cite{Kuziak2013c}, corona product graphs and join graphs \cite{Kuziak2013} and rooted product graphs \cite{Kuziak2013b}. In this paper we study the strong metric dimension of lexicographic product graphs.

We begin by giving some basic concepts and notations. Let $G=(V,E)$ be a simple graph. For two adjacent vertices $u$ and $v$ of $G$ we use the notation  $u\sim v$ and, in this case, we say that $uv$ is an edge of $G$, \emph{i.e.}, $uv\in E$. The complement $G^c$ of $G$ has the same vertex set than $G$ and $uv\in E(G^c)$ if and only if $uv\notin E$. The diameter of $G$ is defined as $$D(G)=\max_{u,v\in V}\{d_G(u,v)\}.$$
If $G$ is not connected, then we will assume that the distance between any two vertices belonging to different components of $G$ is infinity and, thus, its diameter is $D(G)=\infty$. For a vertex $v\in V,$ the set $N_G(v)=\{u\in V:\; u\sim v\}$ is the open neighborhood of $v$ and the set $N_G[v] = N_G(v)\cup \{v\}$ is the closed neighborhood of $v$. Two vertices $x$, $y$ are called \emph{true twins} if $N_G[x] = N_G[y]$. In this sense, a vertex $x$ is a \emph{twin} if there exists $y\ne x$ such that they are true twins. We recall that a set $S$ is a \emph{clique} in $G$, if the subgraph induced by $S$ is isomorphic to a complete graph. The \emph{clique number} of a graph $G$, denoted by $\omega(G)$, is the number of vertices in a maximum clique in $G$. We refer to an $\omega(G)$-set in a graph $G$ as a clique of cardinality $\omega(G)$.

A set $S$ of vertices of $G$ is a \emph{vertex cover} of $G$ if every edge of $G$ is incident with at least one vertex of $S$. The \emph{vertex cover number} of $G$, denoted by $\alpha(G)$, is the smallest cardinality of a vertex cover of $G$. We refer to an $\alpha(G)$-set in a graph $G$ as a vertex cover set of cardinality $\alpha(G)$.

Recall that the largest cardinality of a set of vertices of $G$, no two of which are adjacent, is called the \emph{independence number} of $G$ and is denoted by $\beta(G)$. We refer to an $\beta(G)$-set in a graph $G$ as an independent set of cardinality $\beta(G)$. The following well-known result, due to Gallai, states the relationship between the independence number and the vertex cover number of a graph.

\begin{theorem}{\rm (Gallai's theorem)}\label{th gallai}
For any graph  $G$ of order $n$,
$$\alpha(G)+\beta(G) = n.$$
\end{theorem}

A vertex $u$ of $G$ is \emph{maximally distant} from $v$ if for every $w\in N_G(u)$, $d_G(v,w)\le d_G(u,v)$. If $u$ is maximally distant from $v$ and $v$ is maximally distant from $u$, then we say that $u$ and $v$ are \emph{mutually maximally distant}. The {\em boundary} of $G=(V,E)$ is defined as 
$$\partial(G) = \{u\in V:\; \mbox{exists } v\in V\, \mbox{such that } u,v\mbox{ are mutually maximally distant}\}.$$
We use the notion of strong resolving graph introduced by Oellermann and Peters-Fransen in \cite{Oellermann2007}. The \emph{strong resolving graph}\footnote{In fact, according to \cite{Oellermann2007} the strong resolving graph $G'_{SR}$ of a graph $G$ has vertex set $V(G'_{SR})=V(G)$ and two vertices $u,v$ are adjacent in $G'_{SR}$ if and only if $u$ and $v$ are mutually maximally distant in $G$. So, the strong resolving graph defined here is a subgraph of the strong resolving graph defined in \cite{Oellermann2007} and can be obtained from the latter graph by deleting its isolated vertices.} of $G$ is a graph $G_{SR}$  with vertex set $V(G_{SR}) = \partial(G)$ where two vertices $u,v$ are adjacent in $G_{SR}$ if and only if $u$ and $v$ are mutually maximally distant in $G$.

If it is the case, for a non-connected graph $G$ we will use the assumption that any two vertices belonging to different components of $G$ are mutually maximally distant between them.

It was shown in  \cite{Oellermann2007}  that the problem of finding the strong metric dimension of a graph $G$ can be transformed into the problem of computing the vertex cover number of $G_{SR}$.

\begin{theorem}{\em \cite{Oellermann2007}}\label{th oellermann}
For any connected graph $G$,
$$\dim_s(G) = \alpha(G_{SR}).$$
\end{theorem}

We will use the notation $K_n$,  $C_n$, $N_n$ and $P_n$ for complete graphs,  cycle graphs, empty graphs and path graphs, respectively. In this work, the remaining definitions will be given the first time that the concept appears in the text.

\section{The strong metric dimension of the lexicographic product of graphs}\label{sectlex}

The \textit{lexicographic product} of two graphs $G=(V_1,E_1)$ and $H=(V_2,E_2)$ is the graph $G\circ H$ with vertex set $V=V_1\times V_2$ and two vertices $(a,b),(c,d)\in V$  are adjacent in $G\circ H$ if and only if either $ac\in E_1$, or ($a=c$ and $bd\in E_2$).

Note that the lexicographic product of two graphs is not a commutative operation. Moreover, $G\circ H$ is a connected graph if and only if $G$ is connected. For more information on structure and properties of the lexicographic product of graphs we suggest \cite{Hammack2011}. Nevertheless, we would point out the following known results.

\begin{claim}{\rm \cite{Hammack2011}}\label{basictoolLexicographic} Let $G$ and $H$ be two non-trivial graphs such that $G$ is connected. Then the following assertions hold for any  $a,c\in V(G)$ and $b,d\in V(H)$ such  that $a\ne c$.
\begin{enumerate}[{\rm (i)}]
\item $N_{G\circ H}(a,b)=\left(\{a\}\times  N_H(b)\right)\cup \left( N_G(a)\times  V(H)\right)$.
\item  $d_{G\circ H}((a,b),(c,d)) = d_{G}(a,c)$
\item   $d_{G\circ H}((a,b),(a,d)) = \min \{d_{H}(b,d),2\}$.
\end{enumerate}
\end{claim}

From the next lemmas we can describe the structure of the strong resolving graph of $G\circ H$.

\begin{lemma}\label{lem mmd}
Let $G$ be a connected non-trivial graph and let $H$ be a non-trivial graph. Let $a,b\in V(G)$ such that they are not true twin vertices and let $x,y\in V(H)$. Then $(a,x)$ and $(b,y)$ are mutually maximally distant in $G\circ H$ if and only if $a$ and $b$ are mutually maximally distant in $G$.
\end{lemma}

\begin{proof}
Let $x,y\in V(H)$. We assume that $a,b\in V(G)$ are mutually maximally distant in $G$ and that they are not true twins. 
 First of all, notice that $d_G(a,b)\ge 2$ (if $d_G(a,b)=1$, then to be mutually maximally distant in $G$, they must be true twins).
 Hence, by Claim \ref{basictoolLexicographic} (i) we have that if
$(c,d)\in N_{G\circ H}(b,y)$, then either $c=b$ or $c\in N_G(b)$. In both cases, by Claim \ref{basictoolLexicographic} (ii) we obtain $d_{G\circ H}((a,x),(c,d))=d_G(a,c)\le d_G(a,b)=d_{G\circ H}((a,x),(b,y))$. So, $(b,y)$ is maximally distant from $(a,x)$ and, by symmetry, we conclude that $(b,y)$ and $(a,x)$ are mutually maximally distant in $G\circ H$.

Conversely, assume that $(a,x)$ and $(b,y)$, $a\ne b$, are mutually maximally distant in $G\circ H$. If  $c\in N_G(b)$, then for any $z\in V(H)$ we have $(c,z)\in N_{G\circ H}(b,y)$. Now,  by Claim \ref{basictoolLexicographic} (ii) we obtain $d_G(a,c)=d_{G\circ H}((a,x),(c,z))\le d_{G\circ H}((a,x),(b,y))=d_G(a,b)$.
So, $b$ is maximally distant from $a$ and, by symmetry, we conclude that $b$ and $a$ are mutually maximally distant in $G$.
\end{proof}

\begin{lemma}\label{lemmaTrueTwins}
Let $G$ be a connected non-trivial  graph, let $H$ be a graph of order $n\ge 2$, let $a,b\in V(G)$ be two different true twin vertices and let $x,y\in V(H)$. Then $(a,x)$ and $(b,y)$ are  mutually maximally distant in $G\circ H$ if and only if both, $x$ and $y$, have degree $n-1$.
\end{lemma}

\begin{proof}
If  $x\in V(H)$ has degree $n-1$, then for any  $y\in V(H)$ of degree $n-1$ we have that
$(a,x)$  and $(b,y)$ are true twins in $G\circ H$. Hence, $(a,x)$ and $(b,y)$ are  mutually maximally distant in $G\circ H$.

Now, suppose that  there exists $z\in V(H)-N_H(x)$. Hence, Claim \ref{basictoolLexicographic} (iii) leads to $d_{G\circ H}((a,x),(a,z))=2$. Also, for every $y\in V(H)$, Claim \ref{basictoolLexicographic} (ii) leads to $d_{G\circ H}((a,x),(b,y))=1$. Thus, we conclude that $(a,x)$ and $(b,y)$ are not mutually maximally distant in $G\circ H$.
\end{proof}

In order to present our results we need to introduce some more terminology. Given a graph  $G$, we  define $G^*$ as the graph with vertex set $V(G^*)=V(G)$ such that  two vertices $u,v$ are adjacent in $G^*$ if and only if either $d_G(u,v)\ge 2$ or  $u,v$ are true twins. If a graph $G$ has at least one isolated vertex, then we denote by 
$G_-$ the graph obtained from $G$ by removing all its isolated vertices. In this sense, $G^*_-$ is obtained from $G^*$ by removing all its isolated vertices. Notice that $G^*$ satisfies the following straightforward properties.

\begin{remark}\label{rem G*}
 Let $G$ be a connected graph of diameter $D(G)$, order $n$ and maximum degree $\Delta(G).$
\begin{enumerate}[{\rm (i)}]
\item If $\Delta(G)\le n-2$, then $G^*\cong (K_1+G)_{SR}$. 
\item If $D(G)\le 2$, then $G^*_-\cong G_{SR}$.

\item If $G$ has no true twins, then $G^*\cong G^c$.
\end{enumerate}
\end{remark}

\begin{lemma}\label{lem mmd in copy H}
Let $G$ be a connected non-trivial graph. Let $x,y\in V(H)$ be two distinct vertices of a graph $H$  and let $a\in V(G)$. Then $(a,x)$ and $(a,y)$ are mutually maximally distant vertices in $G\circ H$ if and only if $x$ and $y$ are adjacent in $H^*$.
\end{lemma}

\begin{proof}
By Claim \ref{basictoolLexicographic} (iii),  $ d_{G\circ H}((a,x),(a,y))\le 2$ and, by Claim \ref{basictoolLexicographic} (i), if $c\ne a$, then  $(c,w)\in N_{G\circ H}(a,x)$ if and only if $c\in N_{G}(a)$. Hence,  $(a,x)$ and $(a,y)$ are mutually maximally distant if and only if either $(a,x)$ and $(a,y)$ are true twins in $G\circ H$ or $(a,x)$ and $(a,y)$ are not adjacent in $G\circ H$.

On one hand, by definition of lexicographic product, $(a,x)$ and $(a,y)$ are not adjacent in $G\circ H$ if and only if $x$ and $y$ are not adjacent in $ H$.

On the other hand, by Claim \ref{basictoolLexicographic} (i),  $(a,x)$ and $(a,y)$ are true twins in $G\circ H$ if and only if $x$ and $y$ are true twins in $H$.

Therefore, the result follows.
\end{proof}

\begin{proposition}\label{SR-graph-twins-free}
Let $G$ be a connected  graph of order $n\ge 2$ and let $H$ be a non-complete graph of order $n'\ge 2$.
If $G$ has no true twin vertices, then
$$(G\circ H)_{SR}\cong \left(G_{SR}\circ H^*\right) \cup \bigcup_{i=1}^{n-|\partial(G)|} H^*_-. $$
\end{proposition}

\begin{proof}
We assume that $G$ has no true twin vertices. By Lemmas \ref{lem mmd} and \ref{lem mmd in copy H}, we have the following facts.
\begin{itemize}
\item For any $a\not\in \partial (G)$ it follows that $(G\circ H)_{SR}$ has a subgraph, say $H_a$, induced by $(\{a\}\times V(H))\cap \partial(G\circ H)$ which is isomorphic to $H^*_-$
\item For any $b\in \partial (G)$, we have that $(G\circ H)_{SR}$ has a subgraph, say $H_b$, induced by $(\{b\}\times V(H))\cap \partial (G\circ H)$ which is isomorphic to $H^* $.
\item The set $(\partial(G)\times V(H))\cap \partial (G\circ H)$ induces a subgraph in $(G\circ H)_{SR}$ which is isomorphic to $G_{SR}\circ H^*$.
\item For any $a\not\in \partial (G)$ and any $b\in \partial (G)$ there are no edges of $(G\circ H)_{SR}$ connecting vertices belonging to $H_a$ with vertices belonging to $H_b$.
\item For any different vertices $a_1,a_2\not\in \partial (G)$ there are no edges of $(G\circ H)_{SR}$ connecting vertices belonging to $H_{a_1}$ with vertices belonging to $H_{a_2}$.
\end{itemize}
Therefore, the result follows.
\end{proof}

Figure \ref{no true twins} shows the graph $P_4\circ P_3$ and its strong resolving graph. Notice that $(P_3)_-^*\cong K_2$, $(P_3)^*\cong K_2\cup K_1$ and $(P_4)_{SR}\cong K_2$. So, $(P_4\circ P_3)_{SR}\cong K_2\circ (K_2\cup K_1)\cup K_2\cup K_2$.

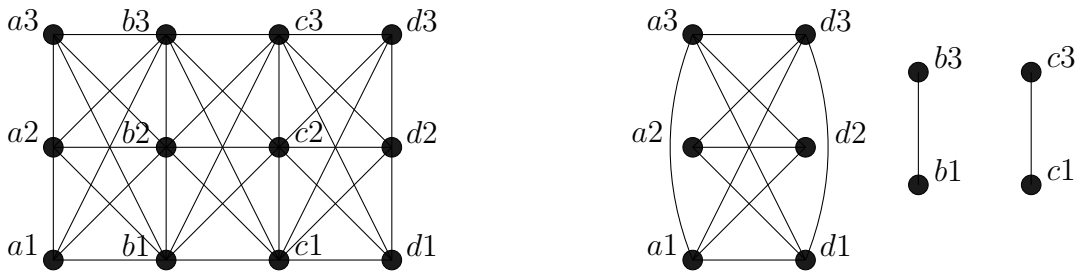
\begin{figure}[h]
\centering
\begin{tikzpicture}
\draw(0,0)--(0,3);
\draw(1.5,0)--(1.5,3);
\draw(3,0)--(3,3);
\draw(4.5,0)--(4.5,3);
\draw(0,0)--(4.5,0);
\draw(0,1.5)--(4.5,1.5);
\draw(0,3)--(4.5,3);

\draw(0,0)--(1.5,1.5);
\draw(0,0)--(1.5,3);

\draw(0,1.5)--(1.5,0);
\draw(0,1.5)--(1.5,3);

\draw(0,3)--(1.5,0);
\draw(0,3)--(1.5,1.5);

\draw(1.5,0)--(3,1.5);
\draw(1.5,0)--(3,3);

\draw(1.5,1.5)--(3,0);
\draw(1.5,1.5)--(3,3);

\draw(1.5,3)--(3,0);
\draw(1.5,3)--(3,1.5);

\draw(3,0)--(4.5,1.5);
\draw(3,0)--(4.5,3);

\draw(3,1.5)--(4.5,0);
\draw(3,1.5)--(4.5,3);

\draw(3,3)--(4.5,0);
\draw(3,3)--(4.5,1.5);

\filldraw[fill opacity=0.9,fill=black]  (0,0) circle (0.13cm);
\filldraw[fill opacity=0.9,fill=black]  (0,1.5) circle (0.13cm);
\filldraw[fill opacity=0.9,fill=black]  (0,3) circle (0.13cm);
\filldraw[fill opacity=0.9,fill=black]  (1.5,0) circle (0.13cm);
\filldraw[fill opacity=0.9,fill=black]  (1.5,1.5) circle (0.13cm);
\filldraw[fill opacity=0.9,fill=black]  (1.5,3) circle (0.13cm);
\filldraw[fill opacity=0.9,fill=black]  (3,0) circle (0.13cm);
\filldraw[fill opacity=0.9,fill=black]  (3,1.5) circle (0.13cm);
\filldraw[fill opacity=0.9,fill=black]  (3,3) circle (0.13cm);
\filldraw[fill opacity=0.9,fill=black]  (4.5,0) circle (0.13cm);
\filldraw[fill opacity=0.9,fill=black]  (4.5,1.5) circle (0.13cm);
\filldraw[fill opacity=0.9,fill=black]  (4.5,3) circle (0.13cm);
\node at (-0.4,0.2) {$a1$ };
\node at (-0.4,1.7) {$a2$ };
\node at (-0.4,3.2) {$a3$ };
\node at (1.1,0.2) {$b1$ };
\node at (1.1,1.7) {$b2$ };
\node at (1.1,3.2) {$b3$ };
\node at (3.4,0.2) {$c1$ };
\node at (3.4,1.7) {$c2$ };
\node at (3.4,3.2) {$c3$ };
\node at (4.9,0.2) {$d1$ };
\node at (4.9,1.7) {$d2$ };
\node at (4.9,3.2) {$d3$ };

\draw(8.5,0)--(10,1.5);
\draw(8.5,0)--(10,3);

\draw(8.5,1.5)--(10,0);
\draw(8.5,1.5)--(10,3);

\draw(8.5,3)--(10,0);
\draw(8.5,3)--(10,1.5);

\draw(11.5,1)--(11.5,2.5);

\draw(13,1)--(13,2.5);

\draw (8.5,0) -- (10,0);
\draw (8.5,1.5) -- (10,1.5);
\draw (8.5,3) -- (10,3);

\filldraw[fill opacity=0.9,fill=black]  (8.5,0) circle (0.13cm);
\filldraw[fill opacity=0.9,fill=black]  (8.5,1.5) circle (0.13cm);
\filldraw[fill opacity=0.9,fill=black]  (8.5,3) circle (0.13cm);
\filldraw[fill opacity=0.9,fill=black]  (10,0) circle (0.13cm);
\filldraw[fill opacity=0.9,fill=black]  (10,1.5) circle (0.13cm);
\filldraw[fill opacity=0.9,fill=black]  (10,3) circle (0.13cm);
\filldraw[fill opacity=0.9,fill=black]  (11.5,1) circle (0.13cm);
\filldraw[fill opacity=0.9,fill=black]  (11.5,2.5) circle (0.13cm);
\filldraw[fill opacity=0.9,fill=black]  (13,1) circle (0.13cm);
\filldraw[fill opacity=0.9,fill=black]  (13,2.5) circle (0.13cm);

\draw (8.5,0) .. controls (8.1,1) and (8.1,2) .. (8.5,3);
\draw (10,0) .. controls (10.4,1) and (10.4,2) .. (10,3);

\node at (8.1,0.2) {$a1$ };
\node at (7.9,1.7) {$a2$ };
\node at (8.1,3.2) {$a3$ };
\node at (11.9,1.2) {$b1$ };
\node at (11.9,2.7) {$b3$ };
\node at (10.4,0.2) {$d1$ };
\node at (10.6,1.7) {$d2$ };
\node at (10.4,3.2) {$d3$ };
\node at (13.4,1.2) {$c1$ };
\node at (13.4,2.7) {$c3$ };

\end{tikzpicture}
\caption{The graph $P_4\circ P_3$ and its strong resolving graph}
\label{no true twins}
\end{figure}

The following well-known result will be a useful tool in determining the strong metric dimension of lexicographic product graphs.

\begin{theorem}{\em \cite{Geller1975}}\label{lem lexicographic coveringNum}
For any graphs $G$ and $H$ of order $n$ and $n'$, respectively, $$\alpha(G\circ H) =n \alpha(H)+n'\alpha(G) -\alpha(G) \alpha(H).$$
\end{theorem}

\begin{theorem}\label{ThGhasNotrueTwins}
Let $G$ be a connected graph of order $n\ge 2$ and let $H$ be a  graph of order $n'\ge 2$. If $G$ has no true twin vertices, then the following assertions hold:
\begin{enumerate}[{\rm (i)}]
\item
If $D(H)\le 2$, then
$\dim_s(G\circ H)=n\cdot \dim_s(H) + n'\cdot \dim_s(G)-\dim_s(G)\dim_s(H).$
\item If $D(H)>  2$, then
$\dim_s(G\circ H)=n\cdot \dim_s(K_1+H) + n'\cdot \dim_s(G)-\dim_s(G)\dim_s(K_1+H).$
\end{enumerate}
\end{theorem}

\begin{proof}
By Theorem \ref{th oellermann} and Proposition \ref{SR-graph-twins-free} we have,
$$\dim_s(G\circ H)=\alpha(G_{SR}\circ H^*) + (n-|\partial(G)|) \alpha (H^*_-)$$
and, by Theorem \ref{lem lexicographic coveringNum} we have
\begin{equation}\label{DimHttfree}
\dim_s(G\circ H)=|\partial(G)| \alpha(H^*) +n'\alpha(G_{SR})-\alpha(G_{SR})\alpha(H^*)+ (n-|\partial(G)|) \alpha (H^*_-).
\end{equation}
 Now, if $D(H)\le 2$, then $\alpha(H^*)=\alpha(H^*_-)=\alpha(H_{SR})$ and, as a result,
$$\dim_s(G\circ H)=n \alpha(H_{SR}) +n'\alpha(G_{SR})-\alpha(G_{SR})\alpha(H_{SR}).$$
Also, and if $D(H)> 2$,
then $\alpha(H^*)=\alpha(H^*_-)=\alpha((K_1+H)_{SR})$, so
$$\dim_s(G\circ H)=n \alpha((K_1+H)_{SR}) +n'\alpha(G_{SR})-\alpha(G_{SR})\alpha((K_1+H)_{SR}).$$
Therefore, by Theorem \ref{th oellermann} we conclude the proof.
\end{proof}
Note that the case where $H$  is non-connected is also considered in Theorem \ref{ThGhasNotrueTwins},
 because we are assuming that if  $H$ is non-connected, then $D(H)=\infty  > 2$.

Now we show some particular examples of graphs $G$ without true twin vertices where $\dim_s(G)$ is easy to compute  or known.

\begin{enumerate}[(1)]
\item For any complete $k$-partite graph $G=K_{p_1,p_2,...,p_k}$ such that  $p_i\ge 2$, $i\in\{1,2,...,k\}$, we have $(G)_{SR}\cong\bigcup_{i=1}^{k}K_{p_i}$. Hence,
$\dim_s(G)=\sum_{i=1}^k(p_i-1).$

\item For any tree $T$  with $l(T)$ leaves, $(T)_{SR}\cong K_{l(T)}$, so $\dim_s(T)=l(T)-1$.

\item The strong resolving graph of any cycle graph  is   $(C_{2k})_{SR}\cong \bigcup_{i=1}^k K_2$ or $(C_{2k+1})_{SR}\cong C_{2k+1}$. So,
$\dim_s(C_{2k})=k$ and $\dim_s(C_{2k+1})=k+1$.

\item The strong resolving graph of any grid graph $P_r\Box P_t$  is $(P_r\Box P_t)_{SR}=K_2\cup K_2$. Thus, $\dim_s(P_r\Box P_t)=2$.

\item For any connected graph $G_1$ of order $n_1$ and any graph $G_2$, the corona graph $ G_1 \odot G_2$ is obtained by taking one copy of $G_1$ and $n_1$ copies of  $G_2$ and joining by an edge the $i$-th vertex of $G_1$ to every vertex of the $i$-th copy of $G_2$. It was shown in \cite{Kuziak2013} that if $n_1\ge 2$ and $G_2$ is a triangle free graph of order $n_2\ge 2$ and  maximum degree $\Delta(H)\le n_2-2$, then $\dim_2(G_1 \odot G_2)=n_1n_2-2$.
\end{enumerate}

Notice that by using Theorem \ref{ThGhasNotrueTwins} (or other ones given throughout the article), and the above known values for several families of graphs, we can obtain directly the strong metric dimension of several combinations of lexicographic product of two graphs. We omit these calculations and leave it to the reader.

According to Theorem \ref{ThGhasNotrueTwins} (i), for any connected graph $G$ without true twin vertices it holds  $\dim_s(G\circ K_{n'})=n(n'-1)+\dim_s(G).$ Now we will show that this formula holds for any connected graph $G$.

\begin{proposition}\label{SR G and complete}
For any connected non-trivial graph $G$ of order $n\ge 2$ and any integer $n'\ge 2$,
$$(G\circ K_{n'})_{SR}\cong (G_{SR}\circ K_{n'}) \cup \bigcup_{i=1}^{n-|\partial(G)|} K_{n'}.$$
\end{proposition}
\begin{proof}
Notice that $(K_{n'})^*\cong K_{n'}$ and, by Lemma \ref{lem mmd in copy H}, for any $a\in V(G)$, the subgraph of $(G\circ K_{n'})_{SR}$ induced by $(\{a\}\times V(K_{n'}))\cap \partial(G\circ K_{n'})$ is isomorphic to $K_{n'}$. Also, from Lemmas \ref{lem mmd} and \ref{lemmaTrueTwins}, the subgraph of $(G\circ K_{n'})_{SR}$ induced by $(\partial(G)\times V(K_{n'}))\cap \partial (G\circ K_{n'})$ is isomorphic to $G_{SR}\circ K_{n'}$. Moreover, for $a\not\in \partial (G)$ and $b\in \partial (G)$ there are not edges of $(G\circ K_{n'})_{SR}$ connecting vertices belonging to $\{a\}\times V(K_{n'})$ with vertices belonging to $\{b\}\times V(K_{n'})$. Therefore, the result follows.
\end{proof}

\begin{theorem}
For any connected non-trivial graph $G$ of order $n\ge 2$ and any integer $n'\ge 2$,
$$\dim_s(G\circ K_{n'})=n(n'-1)+\dim_s(G).$$
\end{theorem}
\begin{proof}
From Theorem \ref{th oellermann} and Proposition \ref{SR G and complete} we have,
$$\dim_s(G\circ K_{n'})=\alpha(G_{SR}\circ K_{n'}) + (n-|\partial(G)|)(n'-1)$$
and, by using Theorem \ref{lem lexicographic coveringNum} and again Theorem \ref{th oellermann} we obtain that
\begin{align*}
\dim_s(G\circ K_{n'})&=|\partial(G)| (n'-1) +n'\alpha(G_{SR})-\alpha(G_{SR})(n'-1)+ (n-|\partial(G)|) (n'-1)\\
&=n(n'-1)+\dim_s(G).
\end{align*}
\end{proof}

We have studied the case in which the second factor in the lexicographic product is a complete graph. Since this product is not commutative, it remains to study the case in which the first factor is a complete graph, which we do at next.

\begin{proposition}\label{SR-Completegraph}
Let $n\ge 2$ be an integer and let $H$ be a graph of order $n'\ge 2$.
If $H$ has maximum degree $\Delta(H)\le n'-2$, then
$$(K_n\circ H)_{SR}\cong  \bigcup_{i=1}^n H^*. $$
\end{proposition}

\begin{proof}
We assume that $H$ has maximum degree $\Delta(H)\le n'-2$.  Notice that $H^*$ has no isolated vertices and, by Lemma \ref{lem mmd in copy H}, for any $a\in V(K_n)$, the subgraph $(K_n\circ H)_{SR}$ induced by $(\{a\}\times V(H))\cap \partial(K_n\circ H)$ is isomorphic to $H^*$.

Also, by Lemma \ref{lemmaTrueTwins}, for any different $a,b\in V(K_n)$ and any $x,y\in V(H)$, the vertices $(a,x)$ and $(b,y)$ are not mutually maximally distant in $K_n\circ H$. Therefore, the result follows.
\end{proof}

\begin{theorem}\label{ThCompletegraphTimesH}
Let $n\ge 2$ be an integer and let $H$ be a graph of order $n'\ge 2$ and  maximum degree $\Delta(H)\le n'-2$.
\begin{enumerate}[{\rm (i)}]
\item If $D(H)=2$, then
$\dim_s(K_n\circ H)=n\cdot \dim_s(H).$
\item If $D(H)> 2$, then
$\dim_s(K_n\circ H)=n\cdot \dim_s(K_1+H).$
\end{enumerate}
\end{theorem}

\begin{proof}
By Theorems \ref{th oellermann} and \ref{SR-Completegraph} we have,
$\dim_s(K_n\circ H)=n\cdot\alpha(H^*)$.
Hence, if $D(H)=2$, then
$\dim_s(K_n\circ H)=n \cdot\alpha(H_{SR})$
and if $D(H)> 2$, then
$\dim_s(K_n\circ H)=n \cdot\alpha((K_1+H)_{SR})$.
Therefore, by Theorem \ref{th oellermann} we conclude the proof.
\end{proof}

For the particular case of empty graphs $H=N_{n'}=(K_{n'})^c$, Theorem \ref{ThCompletegraphTimesH} leads to the next corollary, which is straightforward because $ K_n\circ N_{n'} \cong K_{n',n',...,n'}$, is a complete $n$-partite graph, and so $(K_n\circ N_{n'})_{SR}\cong \bigcup_{i=1}^n K_{n'}$.

\begin{corollary}\label{KnNn}
For any integers $n,n'\ge 2$,
$\dim_s(K_n\circ N_{n'})=n(n'-1).$
\end{corollary}

We define the {\em TF-boundary} of a  non-complete graph $G=(V,E)$ as a set $\partial_{TF}(G) \subseteq \partial(G)$, where $x\in \partial_{TF}(G)$ whenever there exists $y\in \partial (G)$, such that $x$ and $y$ are mutually maximally distant in $G$ and $N_G[x]\ne N_G[y]$ (which means that $x,y$ are not true twins).
The \emph{strong resolving TF-graph} of $G$ is a graph $G_{SRS}$  with vertex set $V(G_{SRS}) = \partial_{TF}(G)$, where two vertices $u,v$ are adjacent in $G_{SRS}$ if and only if $u$ and $v$ are mutually maximally distant in $G$ and $N_G[x]\ne N_G[y]$. Since the strong resolving TF-graph is a subgraph of the strong resolving graph, an instance of the problem of transforming a graph into its strong resolving TF-graph forms part of the general problem of transforming a graph into its strong resolving graph. From \cite{Oellermann2007}, it is known that this general transformation is polynomial. Thus, the problem of  transforming a graph into its strong resolving TF-graph is also polynomial.

 An interesting example of a strong resolving TF-graph is obtained from the corona graph $G\odot K_{n'}$, $n'\ge 2$, where $G$ has order $n\ge 2$. Notice that any two different vertices belonging to any two copies of the complete graph $K_{n'}$ are mutually maximally distant, but if they are in the same copy, then they are also true twins. Thus, in this case $\partial_{TF}(G\odot K_{n'})=\partial(G\odot K_{n'})$, while we have have that $(G\odot K_{n'})_{SR}\cong K_{nn'}$ and $(G\odot K_{n'})_{SRS}$ is isomorphic to a complete $n$-partite graph $K_{n',n',...,n'}$.

\begin{proposition}\label{SR-MaxDgn'-2}
Let $G$ be a connected non-complete   graph of order $n\ge 2$ and let $H$ be a graph of order $n'\ge 2$.
If $H$ has maximum degree $\Delta(H)\le n'-2$, then
$$(G\circ H)_{SR}\cong (G_{SRS}\circ H^*) \cup \bigcup_{i=1}^{n-|\partial_{TF}(G)|} H^*. $$
\end{proposition}

\begin{proof}
We assume that $H$ has maximum degree $\Delta(H)\le n'-2$.  Notice that $H^*$ has no isolated vertices and, by Lemma  \ref{lem mmd in copy H}, for any $a\in V(G)$, the subgraph $(G\circ H)_{SR}$ induced by $(\{a\}\times V(H))\cap \partial(G\circ H)$ is isomorphic to $H^*$.

Also, by Lemma \ref{lemmaTrueTwins}, if two different vertices $a,b$ are true twins in $G$ and $x,y\in V(H)$, then $(a,x)$ and $(b,y)$ are not mutually maximally distant in $G\circ H$. So, from Lemmas \ref{lem mmd} and \ref{lem mmd in copy H} we deduce that the subgraph of $(G\circ H)_{SR}$ induced by $(\partial_{TF}(G)\times V(H))\cap \partial (G\circ H)$ is isomorphic to $G_{SRS}\circ H^*$. Moreover, for $a\not\in \partial_{TF} (G)$ and $b\in \partial_{TF} (G)$ there are no edges of $(G\circ H)_{SR}$ connecting vertices belonging to $\{a\}\times V(H)$ with vertices belonging to $\{b\}\times V(H)$. Therefore,
the result follows.
\end{proof}

Figure \ref{max degree} shows the graph $(K_1+(K_1\cup K_2))\circ P_4$ and its strong resolving graph. Notice that $(P_4)^*\cong P_4$ and $(K_1+(K_1\cup K_2))_{SRS} \cong P_3$. So, $((K_1+(K_1\cup K_2))\circ P_4)_{SR}\cong (P_3\circ P_4)\cup P_4$.

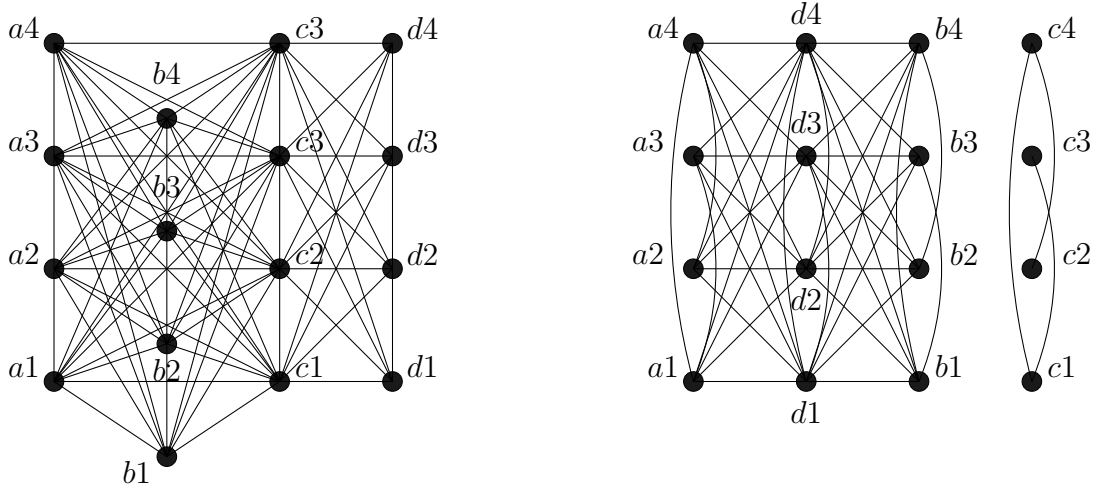
\begin{figure}[h]
\centering
\begin{tikzpicture}
\draw(0,0)--(0,4.5);
\draw(1.5,-1)--(1.5,3.5);
\draw(3,0)--(3,4.5);
\draw(4.5,0)--(4.5,4.5);
\draw(0,0)--(4.5,0);
\draw(0,0)--(1.5,-1);
\draw(1.5,-1)--(3,0);
\draw(0,1.5)--(4.5,1.5);
\draw(0,1.5)--(1.5,0.5);
\draw(1.5,0.5)--(3,1.5);
\draw(0,3)--(4.5,3);
\draw(0,3)--(1.5,2);
\draw(1.5,2)--(3,3);
\draw(0,4.5)--(4.5,4.5);
\draw(0,4.5)--(1.5,3.5);
\draw(1.5,3.5)--(3,4.5);

\draw(0,0)--(1.5,0.5);
\draw(0,0)--(1.5,2);
\draw(0,0)--(1.5,3.5);

\draw(0,1.5)--(1.5,-1);
\draw(0,1.5)--(1.5,2);
\draw(0,1.5)--(1.5,3.5);

\draw(0,3)--(1.5,-1);
\draw(0,3)--(1.5,0.5);
\draw(0,3)--(1.5,3.5);

\draw(0,4.5)--(1.5,-1);
\draw(0,4.5)--(1.5,0.5);
\draw(0,4.5)--(1.5,2);

\draw(0,0)--(3,1.5);
\draw(0,0)--(3,3);
\draw(0,0)--(3,4.5);

\draw(0,1.5)--(3,0);
\draw(0,1.5)--(3,3);
\draw(0,1.5)--(3,4.5);

\draw(0,3)--(3,0);
\draw(0,3)--(3,1.5);
\draw(0,3)--(3,4.5);

\draw(0,4.5)--(3,0);
\draw(0,4.5)--(3,1.5);
\draw(0,4.5)--(3,3);

\draw(1.5,-1)--(3,1.5);
\draw(1.5,-1)--(3,3);
\draw(1.5,-1)--(3,4.5);

\draw(1.5,0.5)--(3,0);
\draw(1.5,0.5)--(3,3);
\draw(1.5,0.5)--(3,4.5);

\draw(1.5,2)--(3,0);
\draw(1.5,2)--(3,1.5);
\draw(1.5,2)--(3,4.5);

\draw(1.5,3.5)--(3,0);
\draw(1.5,3.5)--(3,1.5);
\draw(1.5,3.5)--(3,3);

\draw(3,0)--(4.5,1.5);
\draw(3,0)--(4.5,3);
\draw(3,0)--(4.5,4.5);

\draw(3,1.5)--(4.5,0);
\draw(3,1.5)--(4.5,3);
\draw(3,1.5)--(4.5,4.5);

\draw(3,3)--(4.5,0);
\draw(3,3)--(4.5,1.5);
\draw(3,3)--(4.5,4.5);

\draw(3,4.5)--(4.5,0);
\draw(3,4.5)--(4.5,1.5);
\draw(3,4.5)--(4.5,3);

\filldraw[fill opacity=0.9,fill=black]  (0,0) circle (0.13cm);
\filldraw[fill opacity=0.9,fill=black]  (0,1.5) circle (0.13cm);
\filldraw[fill opacity=0.9,fill=black]  (0,3) circle (0.13cm);
\filldraw[fill opacity=0.9,fill=black]  (0,4.5) circle (0.13cm);
\filldraw[fill opacity=0.9,fill=black]  (1.5,-1) circle (0.13cm);
\filldraw[fill opacity=0.9,fill=black]  (1.5,0.5) circle (0.13cm);
\filldraw[fill opacity=0.9,fill=black]  (1.5,2) circle (0.13cm);
\filldraw[fill opacity=0.9,fill=black]  (1.5,3.5) circle (0.13cm);
\filldraw[fill opacity=0.9,fill=black]  (3,0) circle (0.13cm);
\filldraw[fill opacity=0.9,fill=black]  (3,1.5) circle (0.13cm);
\filldraw[fill opacity=0.9,fill=black]  (3,3) circle (0.13cm);
\filldraw[fill opacity=0.9,fill=black]  (3,4.5) circle (0.13cm);
\filldraw[fill opacity=0.9,fill=black]  (4.5,0) circle (0.13cm);
\filldraw[fill opacity=0.9,fill=black]  (4.5,1.5) circle (0.13cm);
\filldraw[fill opacity=0.9,fill=black]  (4.5,3) circle (0.13cm);
\filldraw[fill opacity=0.9,fill=black]  (4.5,4.5) circle (0.13cm);

\node at (-0.4,0.2) {$a1$ };
\node at (-0.4,1.7) {$a2$ };
\node at (-0.4,3.2) {$a3$ };
\node at (-0.4,4.7) {$a4$ };
\node at (1.1,-1.2) {$b1$ };
\node at (1.5,0.15) {$b2$ };
\node at (1.5,2.6) {$b3$ };
\node at (1.5,4.1) {$b4$ };
\node at (3.4,0.2) {$c1$ };
\node at (3.4,1.7) {$c2$ };
\node at (3.4,3.2) {$c3$ };
\node at (3.4,4.7) {$c3$ };
\node at (4.9,0.2) {$d1$ };
\node at (4.9,1.7) {$d2$ };
\node at (4.9,3.2) {$d3$ };
\node at (4.9,4.7) {$d4$ };

\draw(8.5,0)--(11.5,0);
\draw(8.5,1.5)--(11.5,1.5);
\draw(8.5,3)--(11.5,3);
\draw(8.5,4.5)--(11.5,4.5);

\draw(8.5,0)--(10,1.5);
\draw(8.5,0)--(10,3);
\draw(8.5,0)--(10,4.5);

\draw(8.5,1.5)--(10,0);
\draw(8.5,1.5)--(10,3);
\draw(8.5,1.5)--(10,4.5);

\draw(8.5,3)--(10,0);
\draw(8.5,3)--(10,1.5);
\draw(8.5,3)--(10,4.5);

\draw(8.5,4.5)--(10,0);
\draw(8.5,4.5)--(10,1.5);
\draw(8.5,4.5)--(10,3);

\draw(11.5,0)--(10,1.5);
\draw(11.5,0)--(10,3);
\draw(11.5,0)--(10,4.5);

\draw(11.5,1.5)--(10,0);
\draw(11.5,1.5)--(10,3);
\draw(11.5,1.5)--(10,4.5);

\draw(11.5,3)--(10,0);
\draw(11.5,3)--(10,1.5);
\draw(11.5,3)--(10,4.5);

\draw(11.5,4.5)--(10,0);
\draw(11.5,4.5)--(10,1.5);
\draw(11.5,4.5)--(10,3);


\filldraw[fill opacity=0.9,fill=black]  (8.5,0) circle (0.13cm);
\filldraw[fill opacity=0.9,fill=black]  (8.5,1.5) circle (0.13cm);
\filldraw[fill opacity=0.9,fill=black]  (8.5,3) circle (0.13cm);
\filldraw[fill opacity=0.9,fill=black]  (8.5,4.5) circle (0.13cm);
\filldraw[fill opacity=0.9,fill=black]  (10,0) circle (0.13cm);
\filldraw[fill opacity=0.9,fill=black]  (10,1.5) circle (0.13cm);
\filldraw[fill opacity=0.9,fill=black]  (10,3) circle (0.13cm);
\filldraw[fill opacity=0.9,fill=black]  (10,4.5) circle (0.13cm);
\filldraw[fill opacity=0.9,fill=black]  (11.5,0) circle (0.13cm);
\filldraw[fill opacity=0.9,fill=black]  (11.5,1.5) circle (0.13cm);
\filldraw[fill opacity=0.9,fill=black]  (11.5,3) circle (0.13cm);
\filldraw[fill opacity=0.9,fill=black]  (11.5,4.5) circle (0.13cm);
\filldraw[fill opacity=0.9,fill=black]  (13,0) circle (0.13cm);
\filldraw[fill opacity=0.9,fill=black]  (13,1.5) circle (0.13cm);
\filldraw[fill opacity=0.9,fill=black]  (13,3) circle (0.13cm);
\filldraw[fill opacity=0.9,fill=black]  (13,4.5) circle (0.13cm);

\draw (8.5,0) .. controls (8.1,1.5) and (8.1,3) .. (8.5,4.5);
\draw (8.5,0) .. controls (8.9,1) and (8.9,2) .. (8.5,3);
\draw (8.5,1.5) .. controls (8.9,2.5) and (8.9,3.5) .. (8.5,4.5);
\draw (10,0) .. controls (9.6,1.5) and (9.6,3) .. (10,4.5);
\draw (10,0) .. controls (10.4,1) and (10.4,2) .. (10,3);
\draw (10,1.5) .. controls (10.4,2.5) and (10.4,3.5) .. (10,4.5);
\draw (11.5,0) .. controls (11.1,1.5) and (11.1,3) .. (11.5,4.5);
\draw (11.5,0) .. controls (11.9,1) and (11.9,2) .. (11.5,3);
\draw (11.5,1.5) .. controls (11.9,2.5) and (11.9,3.5) .. (11.5,4.5);
\draw (13,0) .. controls (12.6,1.5) and (12.6,3) .. (13,4.5);
\draw (13,0) .. controls (13.4,1) and (13.4,2) .. (13,3);
\draw (13,1.5) .. controls (13.4,2.5) and (13.4,3.5) .. (13,4.5);

\node at (8.1,0.2) {$a1$ };
\node at (7.9,1.7) {$a2$ };
\node at (7.9,3.2) {$a3$ };
\node at (8.1,4.7) {$a4$ };
\node at (10,-0.4) {$d1$ };
\node at (10,1.1) {$d2$ };
\node at (10,3.45) {$d3$ };
\node at (10,4.9) {$d4$ };
\node at (11.9,0.2) {$b1$ };
\node at (12.1,1.7) {$b2$ };
\node at (12.1,3.2) {$b3$ };
\node at (11.9,4.7) {$b4$ };
\node at (13.4,0.2) {$c1$ };
\node at (13.6,1.7) {$c2$ };
\node at (13.6,3.2) {$c3$ };
\node at (13.4,4.7) {$c4$ };
\end{tikzpicture}
\caption{The graph $(K_1+(K_1\cup K_2))\circ P_4$ and its strong resolving graph}
\label{max degree}
\end{figure}

\begin{theorem}\label{ThDegHlessn-1}
Let $G$ be a connected non-complete  graph of order $n\ge 2$ and let $H$ be a graph of order $n'\ge 2$ and maximum degree $\Delta(H)\le n'-2$.
\begin{enumerate}[{\rm (i)}]
\item
If $D(H)=2$, then
$\dim_s(G\circ H)=n\cdot \dim_s(H) + n'\cdot \alpha(G_{SRS})-\alpha(G_{SRS})\dim_s(H).$
\item If $D(H)> 2$, then
$\dim_s(G\circ H)=n\cdot \dim_s(K_1+H) + n'\cdot \alpha(G_{SRS})-\alpha(G_{SRS})\dim_s(K_1+H).$
\end{enumerate}
\end{theorem}

\begin{proof}
By Theorem \ref{th oellermann} and Proposition \ref{SR-MaxDgn'-2} we have,
$$\dim_s(G\circ H)=\alpha(G_{SRS}\circ H^*) + (n-|\partial_{SR}(G)|) \alpha (H^*)$$ and, by Theorem \ref{lem lexicographic coveringNum}, we have
\begin{equation}\label{DimMaxDgn'-2}
\dim_s(G\circ H)=|\partial(G)| \alpha(H^*) +n'\alpha(G_{SRS})-\alpha(G_{SRS})\alpha(H^*)+ (n-|\partial_{SR}(G)|) \alpha (H^*).
\end{equation}
Now, if $D(H)=2$, then  $\alpha(H^*)=\alpha(H_{SR})$ and, if $D(H)> 2$,
then $\alpha(H^*)=\alpha((K_1+H)_{SR})$. Hence, if $D(H)=2$, then
$$\dim_s(G\circ H)=n \alpha(H_{SR}) +n'\alpha(G_{SRS})-\alpha(G_{SRS})\alpha(H_{SR}),$$
and if $D(H)> 2$, then
$$\dim_s(G\circ H)=n \alpha((K_1+H)_{SR}) +n'\alpha(G_{SRS})-\alpha(G_{SRS})\alpha((K_1+H)_{SR}).$$
Therefore, by Theorem \ref{th oellermann} we conclude the proof.
\end{proof}

Now we consider the case of empty graphs $N_{n'}=(K_{n'})^c$.

\begin{corollary}\label{H=empty}
Let $G$ be a connected non-complete  graph of order $n\ge 2$ and let  $n'\ge 2$ be an integer.
Then
$$\dim_s(G\circ N_{n'})=n(n'-1)+\alpha(G_{SRS}).$$
In particular, if $H$ has no true twin vertices, then $$\dim_s(G\circ N_{n'})=n(n'-1)+\dim_s(G).$$
\end{corollary}

As we can expect, if $G$ has no true twin vertices and $H$  has maximum degree $\Delta(H)\le n'-2$, then both, Theorem \ref{ThGhasNotrueTwins} and Theorem \ref{ThDegHlessn-1}, lead to the same result.

\begin{theorem}
Let $G$ be a connected  graph of order $n\ge 2$ and let $H$ be a graph of order $n'\ge 2$ and maximum degree $\Delta(H)\le n'-2$.
Then the following assertions hold:

\begin{enumerate}[{\rm (i)}]
\item If $H$ has no true twin vertices, then  $$\dim_s(G\circ H)=(n-\alpha(G_{SRS}))(n'- \omega(H))+n'\alpha(G_{SRS}).$$
\item If neither $G$ nor  $H$  have true twin vertices, then  $$\dim_s(G\circ H)=(n-\dim_s(G))(n'- \omega(H))+n'\dim_s(G).$$
\end{enumerate}
\end{theorem}

\begin{proof}
First of all, notice that Theorem \ref{th gallai} leads to $\alpha(H^c)=n'-\beta(H^c)=n'-\omega(H)$. Also, from  $\Delta(H)\le n'-2$  we have $H^*=H^*_-$ and,  if  $H$  has no true twin vertices, then $H^*=H^c$. Hence, (\ref{DimMaxDgn'-2})  leads to (i).
Moreover, if  $G$  has no true twin vertices, then (\ref{DimHttfree}) leads to (ii).
\end{proof}

\section*{Conclusion and open problems}

We have studied the strong metric dimension of lexicographic product graphs $G\circ H$ in the following cases.
\begin{itemize}
\item $H$ is any non-trivial graph and $G$ has no true twins.
\item $G$ is any connected graph and $H$ is a non-trivial graph having maximum degree at most its order minus two.
\end{itemize}
In this sense, it remains to study the case in which $G$ is any connected graph and $H$ is a non-trivial graph having maximum degree equal to its order minus one, which we leave as an open problem.

On the other hand, it can be noticed the very important role which plays the strong resolving graph of a graph into computing its strong metric dimension (this fact can be also noted in the articles \cite{Kuziak2013c,Kuziak2013b,Rodriguez-Velazquez2013a}). According to this interesting usefulness of the strong resolving graph we propose as an open problem, to describe the strong resolving graph of other families of graphs. This problem was already mentioned (but not remarked) in the article \cite{Oellermann2007}, where was open the question of characterizing the class of all graphs having a strong resolving graph isomorphic to a bipartite graph. The motivation for this question is related to the fact that, in this case, the vertex cover number can be computed in polynomial time and, in concordance with Theorem \ref{th oellermann}, also the strong metric dimension. Moreover, is it another interesting application of the strong resolving graph?

\end{document}